\documentclass{amsart}

\usepackage[dvipsnames]{xcolor}

\usepackage[normalem]{ulem}

\usepackage[final]{pdfpages}
\usepackage{mathtools}

\usepackage[makeroom]{cancel}
\usepackage{graphicx}
\usepackage{cancel}
\usepackage{ulem}
\usepackage{setspace}

\usepackage{tikz}
\usetikzlibrary{arrows}
\usetikzlibrary{arrows.meta}
\usetikzlibrary{decorations.pathreplacing}

\usepackage{verbatim}

\usepackage{amsrefs}

\usepackage{graphicx}
\usepackage{mathtools}

\usepackage{graphicx}
\usepackage{float} 
\usepackage{amsfonts}
\usepackage{amsthm}
\usepackage{latexsym}
\usepackage{amsmath}
\usepackage{amssymb}
\usepackage{amscd}
\usepackage{epsf}
\usepackage{tikz}
\usetikzlibrary{matrix,arrows}
\usepackage{enumerate}
\usepackage{eepic}
\usepackage{epic}
\usepackage{amsrefs}

\theoremstyle{definition}
\newtheorem{definition}{Definition}

\theoremstyle{plain}
\newtheorem{theorem}{Theorem}

\newtheorem{corollary}{Corollary}

\newcommand*\z{\ensuremath{\mathbf z}}

\newcommand*\Bell{\ensuremath{\boldsymbol\ell}}

\newcommand*\Bvarphi{\ensuremath{\boldsymbol\varphi}}

\title{Always detectable eigenfunctions on metric graphs}

\author{Pavel Kurasov}

\address{Dept. of Mathematics, Stockholm Univ., 106 91 Stockholm, SWEDEN}

\email{kurasov@math.su.se}

\begin{document}

\maketitle

\begin{abstract}
It is proven following \cite{PlTa} that Laplacians with standard vertex continuous on metric trees and 
with standard and Dirichlet conditions on arbitrary metric graphs possess an infinite sequence of
simple eigenvalues with the eigenfunctions not equal to zero in any non-Dirichlet vertex.
\end{abstract}

\section{Introduction}

Differential operators on metric graphs attract attention of both mathematicians and physicists due
to their unusual but rather elegant spectral properties \cite{BeKu,KuBook,KuSa,Hul2004,Lawniczak2020,LawniczakSciR2021,BaAPPA2021}.
 One such property is
that the eigenfunctions may have support not coinciding with the whole graph, or may just vanish at the
vertices leading to problems when defining nodal domains. Moreover, if one of the eigenfunctions is
vanishing at a vertex $ V_0$, then it is not seen in the Titchmarsh-Weyl $ M$-function associated
with this vertex \cite{KuBook,17Warszawa}:
\begin{equation}
M_{V_0} (\lambda) = - \left( \sum_{n=1}^\infty \frac{| \psi_n (V_0) |^2}{\lambda_n - \lambda} \right)^{-1} 
\end{equation}

This note is inspired by the recent paper \cite{PlTa} devoted to generic eigenfunctions for metric
graphs. By generic eigenfunctions one means the eigenfunctions that are different from zero 
at all vertices and corresponds to simple eigenvalues. For such eigenfunctions the nodal domains
are always clearly defined. It appears that the result can be proven much easier using the language of
multivariate secular polynomials describing the spectrum of metric graphs.

\section{Preliminaries}

Let $ \Gamma $ be a metric graph formed from $ N$ compact edges $ E_n = [x_{2n-1}, x_{2n}] $ of lengths
$ \ell_n = x_{2n} - x_{2n-1} $ joined
together in $ M $ vertices $ V_m $ understood as partitions of the endpoints $ \mathbf V = \{ x_j \}_{j=1}^{2N}. $
The points belonging to different equivalence classes $ V_m $ are identified.
In the Hilbert space $ L_2 (\Gamma) = \oplus_{n=1}^N L_2 (E_n) $
consider the Laplacian $ L = - \frac{d^2}{dx^2} $ defined on the functions from the Sobolev space
$ u \in W_2^2 (\Gamma) = \oplus_{n=1}^N W_2^2 (E_n) $, which 
at every vertex $ V_m $ satisfy either standard (continuity and Kirchhoff) conditions
\begin{equation} \label{vcst}
\left\{
\begin{array}{l}
u \; \mbox{is continuous at $ V_m$},  \\[2mm]
\sum_{x_j \in V_m} \partial u (x_j) = 0 ,
\end{array}
\right.
\end{equation}
or Dirichlet conditions
\begin{equation} \label{vcD}
u(x_j) = 0, \quad x_j \in V_m.
\end{equation}
The directed derivatives $ \partial u(x_j) = - (-1)^j u' (x_j) $ are taken in the direction pointing inside the corresponding edge.
We assume that all vertices with Dirichlet conditions have degree one and call these vertices
{\bf Dirichlet} vertices. All other vertices are called {\bf standard}. The Laplacian is a
non-negative self-adjoint operator with discrete spectrum. It is uniquely determined by the metric graph $ \Gamma $,
provided the Dirichlet vertices are indicated. In what follows we are going to refer to
its spectrum as {\bf graph's spectrum}.

Any solution to the eigenfunction equation on the edge $E_n $ can be written using one of the two equivalent 
representations
\begin{equation} \label{eqaabb}
\begin{array}{ccl}
\psi (x, \lambda) & = & 
a_{2n-1} e^{-ik \vert x-x_{2n-1}\vert} + a_{2n} e^{-ik \vert x-x_{2n}\vert} \\
& = &
b_{2n-1} e^{ik \vert x-x_{2n-1} \vert} + b_{2n} e^{ik \vert x-x_{2n}\vert}.
\end{array} 
\end{equation}
Introducing the $2N$-dimensional vectors $ \vec{a} = \{ a_j \}_{j=1}^{2N}, \vec{b} = \{ b_j \}_{j=1}^{2N}$ one 
gets two linear relations:
$$ \vec{b} = \mathbf S_{\rm e} (k) \vec{a}, \quad \vec{b} = \mathbf S_{\rm v}  \vec{a}, $$
with
\begin{equation}
\mathbf S_{\rm e}  (k) = {\rm diag} \left\{ \left( \begin{array}{cc}
0 & e^{i k \ell_n} \\
e^{i k \ell_n} & 0 
\end{array} \right) \right\}_{n=1}^N
\end{equation}
and $ \mathbf S_{\rm v} $ formed from the vertex scattering matrices $ S_m, \; m =1,2, \dots, M$ given by
$$ S_m =S^{\rm st}_{d_m} = - \mathbf I_{d_m} + \frac{d_m}{2} \left(
\begin{array}{cccc}
1 & 1 & \dots & 1 \\
1 & 1 & \ddots & 1 \\
\vdots & \vdots & \ddots & 1 \\
1 & 1 & \dots & 1 \end{array} \right), \quad S_m= S^D = -1 $$
for standard vertices of degree $d_m$ and for Dirichlet vertices, respectively.
The vertex scattering matrix $ \mathbf S_{\rm v} $ has block-diagonal structure if one permutes
the endpoints collecting together the points belonging to each of the vertices. 
The first relation comes directly from \eqref{eqaabb} and the second one is obtained by substituting
solutions \eqref{eqaabb} into vertex conditions \eqref{vcst} and \eqref{vcD}.

The positive spectrum $ \lambda = k^2 $ can be described as zeroes of the secular function \cite{GuSm,KuNo,KuBook}
\begin{equation}
p(k) := \det \Big( \mathbf S_{\rm e} (k) - \mathbf S_{\rm v} \Big),
\end{equation}
which
is a trigonometric polynomial. 

Following \cite{BaGa,KuSa} let us introduce the secular polynomials in $ N $ complex variables $ \z = (z_1, z_2, \dots, z_N) $
\begin{equation}
P_G (\z) = \det \Big( \mathbf E (\z) - \mathbf S_{\rm v} \Big), \quad \mathbf E (\z) = {\rm diag} \left\{ \left( \begin{array}{cc}
0 & z_n \\
z_n & 0 
\end{array} \right) \right\}_{n=1}^N,
\end{equation}
so that we have
\begin{equation}
p(k) = P (e^{ i k \Bell}), \quad \Bell = (\ell_1, \ell_2, \dots, \ell_n).
\end{equation}
Note that the secular polynomial $ P_G $ is determined by the discrete graph $ G $ corresponding to $ \Gamma $ and by
the set of Dirichlet vertices, but is independent of the metric structure of $ \Gamma$.
The spectrum of the metric graph is given by the intersections of the curve $ e^{ i k \Bell} $ with
the zero set $ \mathit Z_G$ of the secular polynomial
$$ \mathit Z_G = \left\{ \z \in \mathbb T^N \subset \mathbf C^N: P(\z) = 0 \right\} ,$$
where $ \mathbb T = \{ z \in \mathbb C: | z | = 1 \} $ is the unit torus.
We shall also use real coordinates $ \Bvarphi, z_n = e^{i \varphi_n} $, then the spectrum is given by intersections
of the line $ k \Bell $ and  the zero set
$$ \mathbf Z_G = \left\{ \Bvarphi \in \mathbf T^N : P(e^{i\Bvarphi}) = 0 \right\} ,$$
with $ \mathbf T = [0, 2 \pi ] $ being the real torus.
The zero set $ \mathbf Z_G $ in general is an $N-1$-dimensional singular surface on the torus.

\section{Main theorem}

\begin{definition} \label{Def1}
An eigenvalue $ \lambda_n $ and the corresponding eigenfunction are called {\bf generic} if and only if
\begin{enumerate}
\item the eigenvalue is simple;
\item the corresponding eigenfunction does not vanish in any of
the vertices other than at the Dirichlet vertices.
\end{enumerate} 
\end{definition}

\begin{theorem}[Theorems 1 and 2 from \cite{PlTa}] \label{Th1}
Let $ L( \Gamma) $ be the Laplace operator on a finite compact metric graph $ \Gamma$
 with standard and Dirichlet vertex conditions.
Then there exists an infinite sequence of generic eigenfunctions
attaining positive values at the vertices,  provided that either
\begin{itemize}
\item the graph is a tree,
\newline or
\item the set of Dirichlet vertices is non-empty.
\end{itemize}
\end{theorem}
\begin{proof}
It is well-known that the order of positive zeroes of the secular function $ p(k) $ coincides with
the multiplicity of the corresponding eigenvalues \cite{BeKu,KuBook}. Hence to satisfy the first
genericness condition (in Definition \ref{Def1})
it is sufficient that the line $ k \Bell $ crosses $ \mathbf Z_G $ at a regular point.

Assume first that the graph has a Dirichlet vertex, then the line $ k \Bell , \; k > 0 $ for any choice of
the vector $ \Bell$ crosses
first the zero set $ \mathbf Z_G $ at a regular point $ \Bvarphi^1= \Bvarphi^1 (\Bell)$. The eigenvalue $ \lambda^1 = (k^1)^2 $ is
determined by $ k^1 \Bell = \Bvarphi^1$.
The ground state is always simple
(for connected graphs) \cite{KuLMP} and the corresponding eigenfunction can be chosen strictly positive, that is positive everywhere
except at the Dirichlet vertices.
Moreover its dependence on the edge lengths is described by
 Hadamard-type formula (see {\it e.g.} (3.13) in \cite{BeKeKuMuTrans} following
 \cite{Fr05-IJM,CdV,BaLe17})
$$ \frac{d \lambda_1}{d \ell_n} = - \Big( \psi'_1 (x)^2 + \lambda_1 \psi_1 (x)^2 \Big) \vert_{x \in E_n}, $$
connecting the derivative of the ground state energy $ \lambda_1 $ as the length of the edge $ E_n $ changes
to Pr\"ufer's amplitude of
the eigenfunction $ \psi_1 $ on the edge. The derivatives above are zero if and only if
$ \psi'_1 (x) = 0 $ and $ \lambda_1 = 0$, that is if the ground state is a constant function, which is
not the case if Dirichlet vertices are present. It follows that $ \nabla P(\z^1) \neq \mathbf 0 , \z^1 = e^{i \Bvarphi^1}$,
and this inequality holds in a neighbourhood of $ \z^1$. Independently of whether the edge
lengths are rationally dependent or not the crossing points between $ k \Bell $ and $ \mathbf Z_G $
contain an infinite sequence $ \Bvarphi^j $ tending to $ \Bvarphi^1$: 
\begin{enumerate}[-]
\item if the edge lengths are pairwise rationally dependent, then the line $ k \Bell$ passes $ \Bvarphi^1 $
infinitely many times;
\item if the edge lengths are not pairwise rationally dependent, then the intersection points do not
coincide with $ \Bvarphi^1 $, but approaches it as $ j \rightarrow \infty$.
\end{enumerate}
As $ \Bvarphi^j \xrightarrow[j \rightarrow \infty]{ }  \Bvarphi^1 $ the corresponding vector $ \vec{a} = \vec{a} (\Bvarphi^j) $ solving
$$ \big( \mathbf E (e^{ik \Bvarphi^j})   - \mathbf S_{\rm v} \big)  \vec{a} = 0 $$
converges to $ \vec{a} (\Bvarphi^1). $ Similarly
$$ \vec{b} (\Bvarphi^j) = \mathbf S_{\rm v} \vec{a}( \Bvarphi^j) \xrightarrow[j \rightarrow \infty]{ }  \mathbf S_{\rm v} \vec{a}( \Bvarphi^1) =  \vec{b} ( \Bvarphi^1) .$$ 
It follows
that 
$$ \psi_{\lambda^j} (x_i) = a_i ( \Bvarphi^j) + b_i ( \Bvarphi^j)  \xrightarrow[j \rightarrow \infty]{ }  \psi_{\lambda^1} (x_i) \neq 0,$$
where $ \lambda^j = (k^j)^2 $denotes the eigenvalue associated with the crossing point $  \Bvarphi^j = k^j \Bell$ and 
provided $ x_j $ does not belong to a Dirichlet vertex. It follows that, may be taking a subsequence, all
eigenvalues $ \lambda^j $ are not only simple, but the corresponding eigenfunctions
are different from zero at non-Dirichlet vertices.

It remains to consider the case where $ \Gamma $ is a tree with only standard vertex conditions.
In the proof above it was sufficient to have one generic eigenvalue such that the associated crossing point
is a regular point in the zero set $ \mathbf Z_G$.  One possible candidate is the ground state $ \lambda_1 = 0 $ and 
the
eigenfunction identically equal to $1 $. The corresponding Pr\"ufer amplitude is zero and we cannot conclude that
the point $  \Bvarphi = \mathbf 0 $ is a regular point in $ \mathbf Z_G$. Consider the corresponding equilateral
tree $ \mathbf T $ with all edge lengths equal to $ \ell_n = \ell$. The spectrum is then periodic with period
$ \frac{2 \pi}{\ell} $. Consider the eigenfunction associated with the eigenvalue $ \left( \frac{2 \pi}{\ell} \right)^2 $, its
multiplicity is equal to $ 1 $ \cite{KuJFA}.\footnote{The multiplicity of the eigenvalue   $ \left( \frac{2 \pi}{\ell} \right)^2 $ for
equilateral graphs in general is equal to $ 1+ \beta_1$, where $ \beta_1 $ is the first Betti number - the number of cycles
in a graph. In the case of trees $ \beta_1 = 0 $.} This function attains $1$ at all vertices  (the same value as the ground state
eigenfunction)
and
the corresponding eigenvalue is non-zero, hence Pr\"ufer amplitudes are different from zero. It follows that
the point $  \Bvarphi = (2\pi, 2\pi, \dots, 2 \pi) $ and hence the point $  \Bvarphi = \mathbf 0 $ is a regular point in $ \mathbf Z_G$.
It remains to repeat the arguments used in the first part of the proof: for any choice of the edge lengths there is a 
sequence $  \Bvarphi^j $ of intersection points between $ k \Bell $ and $ \mathbf Z_G$ approaching $  \Bvarphi^1 = \mathbf 0 $.
The corresponding eigenvalues are simple and the eigenfunction values at each vertex approach $1$. Taking a subsequence, if necessary,
we get an infinite sequence of generic eigenvalues.
\end{proof}

To our opinion the new proof explains the reason why conditions in the theorem are needed. Our proof is based on
the existence of a single generic eigenfunction. The ground state eigenfunction is a good candidate since we know that
it is always generic. The proof goes well for graphs with Dirichlet points, but for graphs with standard conditions one needs to
prove that the hypersurface $ \mathbf Z_G $ is regular in a neighbourhood of $ \mathbf 0 $. If the graph has cycles, then
the point $ \mathbf 0 $ is not regular $ \nabla P_G(\mathbf 1) = 0 $.  It might happen that $ P_G $ is a square of the first degree
polynomial like in the case of the cycle graph: $ P (z_1) = (z_1-1)^2$, then the set $ \mathbf Z $ is a smooth surface, but
all non-zero eigenvalues are double. 

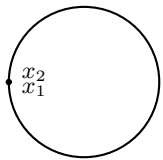
\begin{figure}[H]
\centering
\begin{tikzpicture}[scale=0.5]

\draw[thick] (2,0) circle (2);

\filldraw [black] (0,0) circle (2pt);

\node at (0.1,-0.2) [anchor=west] {\small$x_1$};

\node at (0.1,0.2) [anchor=west] {\small$x_2$};

 \end{tikzpicture}
\caption{The circle graph.}
\label{Fig2.4}
\end{figure}

In fact the circle graph provides a counterexample as it is impossible to
find an infinite sequence of {\bf simple} eigenvalues.

The theorem has another interesting implication
\begin{corollary}
Under conditions of Theorem \ref{Th1} there always exists an infinite sequence of
simple eigenvalues with an even/odd number of nodal domains, provided
the Euler characteristic $ \chi = M-N $ is even/odd, respectively.
\end{corollary}

\end{document}